\documentclass{article}

\usepackage{amssymb}
\usepackage{latexsym}
\usepackage{theorem}
\usepackage[all]{xy}

\newenvironment{proof}{\par \noindent{\bf Proof: }}{\hspace{\stretch{1}} $\Box$ \par \mbox{}}
\newcommand{\noproof}{\hspace{\stretch{1}} $\Box$}

\newtheorem{theorem}{Theorem}[section]
\newtheorem{proposition}[theorem]{Proposition}
\newtheorem{lemma}[theorem]{Lemma}
\newtheorem{corollary}[theorem]{Corollary}
{\theorembodyfont{\rmfamily}
\newtheorem{definition}[theorem]{Definition}
\newtheorem{example}[theorem]{Example}
}
\newenvironment{theorem*}{\par \medskip \noindent{\bf Theorem }}{\par \mbox{}}
\newenvironment{lemma*}{\par \medskip \noindent{\bf Theorem }}{\par \mbox{}}

\newcommand{\Map}{\mathop{\rm Map}}

\newcommand{\R}{\mathbb{R}}

\newcommand{\Z}{\mathbb{Z}}
\newcommand{\N}{\mathbb{N}}

\newcommand{\Cay}{\rm \mathop{Cay}}
\newcommand{\Supp}{\rm \mathop{Supp}}
\newcommand{\id}{\mathop{id}}

\title{An Analytic Novikov Conjecture for Semigroups}

\author{Paul D. Mitchener}

\begin{document}

\maketitle

\section*{Abstract}

In this article we formulate a version of the analytic Novikov conjecture for semigroups rather than groups, and show that the descent argument from coarse geometry generalises effectively to this new situation.

\section{Introduction}

For the purposes of this article, a {\em semigroup} is a set $P$ equipped with an associative binary operation $P\times P\rightarrow P$, such that we have a {\em unit element} $e\in P$ where $pe=ep =p$ for all $p\in P$, and the {\em left cancellation property} holds, that is to say $pq=pr$ implies $q=r$ for all $p,q,r \in P$.  Note that the left cancellation property tells us that the unit element $e$ is unique.

In \cite{Li}, both the reduced and maximal $C^\ast$-algebras associated to a semigroup are defined, issues associated to amenability examined, and $K$-theory groups computed.  The computations of $K$-theory groups lead to a natural question, namely whether a version of the Baum-Connes conjecture (see for instance \cite{BCH}) could be formulated for semigroups.

In this paper, we make a first step towards such a conjecture, formulating an {\em analytic assembly map} $\beta \colon K_n^P (EP)\rightarrow K_n C^\ast_r (P)$, where $C^\ast_r (P)$ is the reduced $C^\ast$-algebra of the semigroup $P$, and $EP$ is the classifying space for free $P$-actions.  We conjecture that this map is injective for torsion-free semigroups.

We also show that the descent argument from the coarse Baum-Connes conjecture, as explained for example in \cite{Roe1}, or in general in \cite{Mitch13}, still works in the semigroup case.  Thus the analytic Novikov conjecture holds for semigroups where the space $EP$ is a finite $P$-$CW$-complex, and has a compatible coarse structure where the coarse Baum-Connes conjecture is satisfied.

The descent argument works in the same way as it does for groups, but to carry it out we need to generalise parts of the general theory of equivariant homology for group actions to the semigroup case.  These generalisations are fortunately straightforward, and the details can be found in sections 4 and 5.

We conclude the article by looking at some simple examples where the descent argument applies.

\section{Semigroup Actions}

Let $P$ be a semigroup.  Let $X$ be a set.  A {\em left $P$-action} on $X$ is a map $P\times X\rightarrow X$, written $(p,x)\mapsto px$, such that $p(qx) = (pq)x$ for all $p,q\in P$ and $x\in X$.  

Similarly, a {\em right $P$-action} on $X$ is a map $X\times P\rightarrow X$, written $(x,p)\mapsto xp$, such that $(xp)q = x(pq)$ for all $p,q\in P$ and $x\in X$.

For a set $X$ equipped with a left $P$-action, and a subset $A\subseteq X$, we write
\[ pA = \{ pa \ | \ a\in A \} \qquad PA = \bigcup_{p\in P} PA . \]

Given sets $X$ and $Y$ with left $P$-actions, a map $f\colon X\rightarrow Y$ is called {\em equivariant} if $f(px) = pf(x)$ for all $x\in X$ and $p\in P$.  We similarly talk about equivariant maps between sets equipped with right $P$-actions.

A {\em $P$-space} is a topological space equipped with a continuous left $P$-action.  Just as for groups acting on spaces, we distinguish certain types of $P$-spaces.  Given $P$-spaces $X$ and $Y$, we write $\Map_P (X,Y)$ to denote the set of all continuous equivariant maps from $X$ to $Y$.  It is a topological space, with the compact open topology.

\begin{definition}
Let $X$ be a $P$-space.  Then we call $X$:

\begin{itemize}

\item {\em Free} if for all $x\in X$, there is an open neighbourhood $U\ni x$ such that $pU\cap U =\emptyset$ for all $p\in P\backslash \{ e \}$, where $e$ is the unit element of $P$.

\item {\em Cocompact} if there is a compact subset $K\subseteq X$ such that $X=PK$.

\end{itemize}

We call an equivariant continuous map $f\colon X\rightarrow Y$ {\em proper} if whenever $Z\subseteq Y$ is cocompact, the inverse image $f^{-1}[Z]\subseteq X$ is also cocompact.

\end{definition}

\begin{example}
We call a subset  $S\subseteq P$ a {\em generating set} if every element of $P$ is a product of elements of $S$.  The {\em Cayley graph} $\Cay (P;S)$ is the graph with set of vertices $P$, where $p,q\in P$ are joined by an edge if and only if $p=sq$ or $q=sp$.

We give Cayley graph is given the path-length metric where each edge has length $1$.  There is a free $P$-action on the space $\Cay (P;S)$ defined by left-multiplication on the vertices, and extending to be linear on the edges.  The $P$-space $\Cay (P;S)$ is cocompact if $S$ is finite.
\end{example}

\begin{example} \label{ij}
The {\em infinite join} (see \cite{Mil6}), $P\ast P\ast P \ast \cdots$ of countably many copies of the semigroup $P$ is a free and weakly contractible $P$-space.
\end{example}

Let $X$ and $Y$ be metric spaces.  Recall (see for instance \cite{Roe6}) that a (not necessarily continuous) map $f\colon X\rightarrow Y$ is called a {\em coarse map} if:

\begin{itemize}

\item For all $R>0$ there exists $S>0$ such that if $d(x,y)<R$, then $d(f(x),f(y))<S$.

\item Let $B\subseteq Y$ be bounded.  Then $f^{-1}[B] \subseteq X$ is also bounded.

\end{itemize}

A {\em coarse $P$-space} is a proper metric space $X$ equipped with a $P$-action such that for each $p\in P$, the map $p\colon X\rightarrow X$ is both coarse and continuous.

Note that for a generating set $S$, the Cayley graph $\Cay (P;S)$ is an example of a coarse $P$-space.

\begin{definition}
Let $P$ be a semigroup.  We call an equivalence relation, $\sim$, on $P$ a {\em left congruence} if whenever $p\sim q$ and $p,q,r\in P$, we have $rp\sim rq$.
\end{definition}

Observe that if we have a left congruence, $\sim$, we have a left $P$-action on the set of equivalence classes $P/\sim$ defined by writing $p([q]) = [pq]$, where $[r]$ is the equivalence class containing an element $r\in P$.  The quotient $P/\sim$ can be considered a $P$-space with the discrete topology.

\begin{definition}
A {\em homogeneous $P$-space} is a $P$-space $X$ such that there is an equivariant homeomorphism $X\rightarrow P/\sim$ for some left congruence $\sim$.
\end{definition}

We now define a class of $P$-spaces of particular importance to us, called {\em $P$-$CW$-complexes}.  Firstly, write
\[ D^{n+1} = \{ (x_0,\ldots ,x_n )\in \R^{n+1} \ |\ x_0^2 +\ldots + x_n^2 \leq 1 \} \]
and
\[ S^n = \{ (x_0,\ldots ,x_n )\in \R^{n+1} \ |\ x_0^2 +\ldots + x_n^2 = 1 \} . \]

Note that $S^n \subseteq D^{n+1}$.  An {\em $n$-dimensional $P$-cell} is a $P$-space of the form $X\times D^n$, where $X$ is a homogeneous $P$-space, and $P$ acts trivially on $D^n$.  

Given a $P$-space $Y$ and $P$-cell $X\times D^n$ equipped with a continuous equivariant map $f\colon S^{n-1}\times X\rightarrow Y$, we can form a $P$-space
\[ (X\times D^n) \cup_{X\times S^{n-1}} Y = \frac{(X\times D^n)\amalg Y}{\sim} \]
where $(x,s)\sim f(x,s)$ if $(x,s)\in X\times S^{n-1}$.

We call the $P$-space $(X\times D^n) \cup_{X\times S^{n-1}} Y$ the space obtained from $Y$ by {\em attaching} the $P$-cell $X\times D^n$ by the map $f$.

\begin{definition}
A {\em finite $P$-$CW$-complex} is a $P$-space $X$ together with a sequence of subspaces
\[ X^0 \subseteq X^1 \subseteq \cdots \subseteq X^n = X \]
such that:

\begin{itemize}

\item The space $X^0$ is a finite disjoint union of homogeneous $P$-spaces.

\item The space $X^k$ is equivariantly homeomorphic to the space obtained from $X^{k-1}$ by attaching finitely many $k$-dimensional $P$-cells.

\end{itemize}

\end{definition}

The above sequence $X^0 \subseteq X^1 \subseteq \cdots \subseteq X^n =X$ is called a {\em $CW$-decomposition} of $X$.

Note that any finite $P$-$CW$-complex is cocompact.  The following is fairly clear.

\begin{proposition}
Let $X$ be a finite $P$-$CW$-complex.  Then $X$ is free if and and only if it has a $CW$-composition in which, for all $k$, every $k$-dimensional $P$-cell takes the form $P\times D^n$.
\noproof
\end{proposition}

\section{The Coarse Baum-Connes Conjecture}
Let $X$ be a proper metric space.  Recall that a Hilbert space $H$ is called an {\em $X$-module} if the $C^\ast$-algebra of bounded linear operators ${\mathcal L}(H)$ is equipped with a $\ast$-homomorphism $\rho \colon C_0 (X)\rightarrow {\mathcal L}(H)$.  

Let ${\mathcal K}(H)$ be the $C^\ast$-algebra of compact operators on $H$.  Then we call an $X$-module $H$ {\em ample} if $\overline{\rho [C_0 (X)] H} =H$ and $\rho [C_0 (X)]\cap {\mathcal K}(H) = \{ 0 \}$.  

\begin{definition}
Let $H$ be an $X$-module, and let $T\in {\mathcal L}(H)$.  Then:

\begin{itemize}

\item We call $T$ {\em locally compact} if $\rho (f)T ,T\rho (f) \in {\mathcal K}(H)$ for all $f\in C_0 (X)$.

\item We call $T$ {\em pseudolocal} if $\rho (f)T - T\rho (f) \in {\mathcal K}(H)$ for all $f\in C_0 (X)$.

\item We define the {\em support} of $T$, $\Supp (T)\subseteq X\times X$, to be the set of pairs $(x,y)\in X\times X$ suc that for all open sets $U\ni x$ and $V\ni y$, we have $f\in C_0 (U)$ and $g\in C_0 (V)$ such that $\rho (f) T \rho (g) \neq 0$.

\item We call $T$ {\em controlled} if the support $\Supp (T)$ is contained in a neighbourhood of the diagonal, $\Delta_R = \{ (x,y)\in X\times X \ |\ d(x,y) <R \}$, for some $R>0$.

\end{itemize}

\end{definition}

\begin{definition}
Let $H$ be an ample $X$-module.  Then we define $D^\ast (X)$ to be the smallest $C^\ast$-subalgebra of ${\mathcal L}(H)$ containing all pseudolocal and controlled operators.

We define $C^\ast (X)$ to be the smallest $C^\ast$-subalgebra of ${\mathcal L}(H)$ containing all locally compact and controlled operators.
\end{definition}

Now, $C^\ast (X)$ is a $C^\ast$-ideal in $D^\ast (X)$, so we have a short exact sequence
\[ 0 \rightarrow C^\ast (X) \rightarrow D^\ast (X) \rightarrow \frac{D^\ast (X)}{C^\ast (X)} \rightarrow 0 . \]

Further, as shown in \cite{HR1,Roe1}, the $K$-theory group of the quotient, $K_n \left( \frac{D^\ast (X)}{C^\ast (X)} \right)$, is isomorphic to the $K$-homology group $K_{n-1} (X)$, and the $K$-theory group $K_n C^\ast (X)$ does not depend on a particular choice of $X$-module.  Thus, looking at the boundary maps in the long exact sequence of $K$-theory groups (see for example \cite{RLL,W-O}), we obtain a map
\[ \alpha \colon K_\ast (X)\rightarrow K_\ast C^\ast (X) \]
called the {\em coarse assembly map}.

The {\em coarse Baum-Connes conjecture} asserts that this map is an isomorphism whenever the space $X$ has bounded geometry and is uniformly contractible; we refer the reader again to \cite{HR1,Roe1} for details, including precisely what the terms bounded geometry and uniformly contractible.

The coarse Baum-Connes conjecture is known to be true for a vast number of spaces, perhaps most notably bounded geometry and uniformly contractible spaces which can be uniformly embedded in Hilbert space (see \cite{Yu2}), but, as shown in \cite{HLS}, is false in general.

\section{Equivariant Homology} \label{EH}

\begin{definition}
Let $f,g\colon X\rightarrow Y$ be equivariant maps between $P$-spaces.  A {\em $P$-homotopy} between $f$ and $g$ is an equivariant continuous map $H\colon X\times [0,1]\rightarrow Y$ such that $H(-,0) =f$ and $H(-,1)=g$.
\end{definition}

Above, the space $X\times [0,1]$ is given the $P$-action defined by the formula $p(x,t) = (px,t)$ where $p\in P$, $x\in X$ and $t\in [0,1]$.

If a $P$-homotopy exists between maps $f$ and $g$, we call them {\em $P$-homotopic}, and write $f\simeq_P g$.  The notion of being $P$-homotopic is an equivalence relation.

A continuous equivariant map $f\colon X\rightarrow Y$ is called a {\em $P$-homotopy equivalence} if there is a continuous equivariant map $g\colon Y\rightarrow X$ such that $g\circ f\simeq_P \id_X$ and $f\circ g\simeq_P \id_Y$.  We write $X\simeq_P Y$ when a $P$-homotopy equivalence $X\rightarrow Y$ exists.

\begin{definition}
A {\em locally finite $P$-homology theory}, $h_\ast^P$, consists of a sequence of functors, $h_n^P$, from the category of $P$-spaces and proper equivariant maps to the category of abelian groups satisfying the following axioms.

\begin{itemize}

\item Let $f,g\colon X\rightarrow Y$ be equivariant continuous maps that are properly $P$-homotopic.  Then the maps $f_\ast , g_\ast \colon h_n^P (X) \rightarrow h_n^P (Y)$ induced by the functor $h_n^P$ are equal for all $n$.

\item Let $X=A\cup B$ be a $P$-space, where $A,B\subseteq X$ are open, and $PA\subseteq A$, $PB\subseteq B$.  Consider the inclusions $i\colon A\cap B\hookrightarrow A$, $j\colon A\cap B\hookrightarrow B$, $k\colon A\hookrightarrow X$ and $l\colon B\hookrightarrow X$.  Let $\alpha = (i_\ast ,-j_\ast )\colon h_n^P (A\cap B)\rightarrow h_n^P (A)\oplus h_n^P (B)$ and $\beta = k_\ast + l_\ast \colon h_n^P (A)\oplus h_n^P (B)\rightarrow h_n^P (X)$.  Then we have natural maps $\partial \colon h_n^P (X)\rightarrow h_{n-1}^P (A\cap B)$ fitting into a long exact sequence
\[ \rightarrow h_n^P (A\cap B) \stackrel{\alpha}{\rightarrow} h_n^P (A)\oplus h_n^P (B) \stackrel{\beta}{\rightarrow} h_n^P (X) \stackrel{\partial}{\rightarrow} h_{n-1}^P (A\cap B) \rightarrow . \]

\item $h_n (\emptyset ) = \{ 0 \}$ for all $n$.

\end{itemize}

\end{definition}

We calll the first of these axioms {\em homotopy invariance}.  The long exact sequence in the second axiom is called the {\em Mayer-Vietoris sequence} associated to the decomposition $X=A\cup B$.

We can also talk about locally finite $P$-homology theories on subcategories of the category of $P$-spaces and proper equivariant maps, for instance on the category of free $P$-spaces.

We now show that knowing a locally finite $P$-homology theory for homogeneous $P$-spaces uniquely determines it for finite $P$-$CW$-complexes.

\begin{lemma} \label{HM1}
Let $Y$ be a $P$-space, let $f\colon X\times S^{n-1}\rightarrow Y$ be a proper equivariant continuous map, and let $Z= (X\times D^n)\cup_{X\times S^{n-1}} Y$.  Then we have a natural long exact sequence
\[ \rightarrow h_n^P (X\times S^{n-1}) \rightarrow h_n^P (X)\oplus h_n^P (Y) \rightarrow h_n^P (Z)\stackrel{\partial}{\rightarrow} h_{n-1}^P (X\times S^{n-1}) \rightarrow \]

Further, the map $\partial$ arises from a Mayer-Vietoris sequence associated to a decomposition of $Z$.
\end{lemma}

\begin{proof}
Let $\pi \colon (X\times D^n)\amalg Y \rightarrow Z$ be the quotient map.  We can choose open neighbourhoods, $U$ and $V$, of $\pi [X\times D^n]$ and $\pi [Y]$ respectively such that:

\begin{itemize}

\item $PU\subseteq U$, $PV\subseteq V$.

\item $U\simeq_P X\times D^n\simeq_P X$.

\item $V\simeq_P Y$.

\item $U\cap V\simeq_P X\times S^{n-1}$.

\end{itemize}

Applying the Mayer-Vietories sequence of the decomposition $X=U\cup V$, along with homotopy invariance, we get a long exact sequence
\[ \rightarrow h_n^P (X\times S^{n-1}) \rightarrow h_n^P (X)\oplus h_n^P (Y) \rightarrow h_n^P (Z)\stackrel{\partial}{\rightarrow} h_{n-1}^P (X\times S^{n-1}) \rightarrow \]
\end{proof}

\begin{definition}
Let $h_\ast^P$ and $k_\ast^P$ be locally finite $P$-homology theories.  A {\em natural transformation} $\tau \colon h_\ast^P \rightarrow k_\ast^P$ is a sequence of natural transformations $\tau \colon h_n^P \rightarrow k_n^P$ that preserves Mayer-Vietoris sequences.
\end{definition}

\begin{lemma} \label{HM2}
Let $X$ be a $P$-space, and let $\tau \colon h_\ast^P \rightarrow k_\ast^P$ be a natural transformation between $P$-homology theories such that the maps $\tau \colon h_n^P (X) \rightarrow k_n^P (X)$ are isomorphisms.

Then the maps $\tau \colon h_n^P (X\times S^k)\rightarrow k_n^P (X\times S^k)$ are all isomorphisms.
\end{lemma}

\begin{proof}
Observe
\[ X\times S^0 = X_1 \amalg X_2 \]
where $X_1$ and $X_2$ are both equivariantly homeomorphic to $X$.  Certainly $X_1\cap X_2 = \emptyset$, so $h_n^P (X_1 \cap X_2 ) =0$ for all $n$, and the Mayer-Vietoris sequence tells us that $h_n^P (X\times S^0 ) = h_n^P (X)\oplus h_n^P (X)$.  Similarly, $k_n^P (X\times S^0 ) = k_n^P (X)\oplus k_n^P (X)$.

It follows immediately that the map $\tau \colon h_n^P (X\times S^0 ) \rightarrow k_n^P (X\times S^0)$ is an isomorphism.

Now suppose the map $\tau \colon h_n^P (X\times S^{k-1})\rightarrow k_n^P (X\times S^{k-1})$ is an isomorphism for all $n$.  We can write $S^k = A\cup B$, where $A\cong D^n$, $B\cong D^n$ and $A\cap B\simeq S^{k-1}$, so $X\times A \simeq_P B\times A\simeq_P X$.  Then we have a commutative diagram of Mayer-Vietoris sequences
\[ \arraycolsep 0pt \small \begin{array}{ccccccccc} 
h_n^P(X\times S^{k-1}) & \rightarrow & h_n^P (X)\oplus h_n (X) & \rightarrow & h_n^P(S^k\times X) & \rightarrow & h_{n-1}^P (S^{k-1} \times X) & \rightarrow & h_{n-1}^P (X)\oplus h_{n-1}^P(X) \\
\downarrow & & \downarrow & & \downarrow & & \downarrow & & \downarrow \\
k_n^P(X\times S^{k-1}) & \rightarrow & k_n^P (X)\oplus k_n (X) & \rightarrow & k_n^P(S^k\times X) & \rightarrow & k_{n-1}^P (S^{k-1} \times X) & \rightarrow & k_{n-1}^P (X)\oplus h_{n-1}^P(X) \\
\end{array} \]

By the five lemma, we see the map $\tau \colon h_n^P (X\times S^k)\rightarrow k_n^P (X\times S^k)$ is an isomorphism for all $n$.  The desired result now follows by induction
\end{proof}

\begin{theorem}
Let $\tau \colon h_\ast \rightarrow k_\ast$ be a natural transformation of $P$-homology theories such that $\tau \colon h_m^P (X)\rightarrow k_m^P (X)$ is an isomorphism whenever $X$ is a homogeneous $P$-space.  Then $\tau \colon h_m^P (Z) \rightarrow k_m^P (Z)$ is an isomorphism whenever $Z$ is a finite $P$-$CW$-complex.
\end{theorem}

\begin{proof}
Let $Z$ be a finite $P$-$CW$-complex.   Then we have subsets
\[ Z^0 \subseteq Z^1 \subseteq \cdots \subseteq Z^n = Z \]
where $Z^0$ is a finite disjoint union of homogeneous $P$-spaces, and $Z^k$ is equivariantly homeomorphic to the space obtained from $Z^{k-1}$ by attaching finitely many $k$-dimensional $P$-cells.

Certainly, the map $\tau \colon h_m^P (Z^0)\rightarrow k_m^P (Z^0)$ is an isomorphism for all $m$.  

Let $Y$ be a $P$-space such that the map $\tau \colon h_m^P (Y)\rightarrow h_m^P (Y)$ is an isomorphism for all $m$.  Suppose we have an attaching map $f\colon X\times S^{n-1}\rightarrow Y$, for a homogeneous $P$-space $X$.  Let $Y' = (X\times D^n)\cup_{X\times S^{n-1}}Y$.  Then it follows by lemma \ref{HM1}, lemma \ref{HM2} and the five lemma that the map $\tau \colon h_m^P (Y')\rightarrow h_m^P (Y')$ is an isomorphism for all $m$.

But this proves the desired result by induction.
\end{proof}

The following is proved similarly.

\begin{theorem} \label{freeiso}
Let $\tau \colon h_\ast \rightarrow k_\ast$ be a natural transformation of $P$-homology theories such that $\tau \colon h_m^P (P)\rightarrow k_m^P (P)$ for all $m$.  Then $\tau \colon h_m^P (Z) \rightarrow k_m^P (Z)$ is an isomorphism whenever $Z$ is a free finite $P$-$CW$-complex.
\end{theorem}

\section{Homotopy Fixed Point Sets}

Let $E$ be a free $P$-space.  Given an equivariant continuous map $f\colon X\rightarrow Y$, we have an induced map $f_\ast \colon \Map_P (E,X)\rightarrow \Map_P (E,Y)$ defined by the formula $f_\ast (g)(\lambda ) = f(g(\lambda ))$ where $g\in X^{hP}$ and $\lambda \in EP$.

\begin{proposition} \label{weco}
Let $E$ be a free $P$-space.  Let $X$ be (non-equivariantly) weakly contractible.  Then the space $\Map_P (E,X)$ is also weakly contractible.
\end{proposition}

\begin{proof}
By definition of freeness, the quotient map $\pi \colon E\rightarrow E/P$ is a covering map, and therefore a fibration, with fibre $P$.  Hence the induced map $\pi^\ast \colon \Map_P (E/P,X)\rightarrow \Map_P (E,X)$ defined by the formula $\pi^\ast (f) = f\circ \pi$ is a cofibration, with cofibre $\Map_P (P,X)$.

Now, since $X$ is weakly contractible, so are the spaces $\Map_P (P,X) \cong X$ and $\Map_P (E/P,X) = \Map (E/P , X)$.  Looking at the long exact sequence of homotopy groups associated to the cofibration $\pi^\ast$, it follows that the space $\Map_P (E,X)$ is also weakly contractible.
\end{proof}

The following immediately follows from the above by looking at mapping cones.

\begin{corollary} \label{maco}
Let $X$ and $Y$ be $P$-spaces, and let $f\colon X\rightarrow Y$ be an equivariant map that is (non-equivariantly) a weak equivalence.  Then the induced map $f_\ast \colon \Map_P (E,X)\rightarrow \Map_P (E,Y)$ is a weak equivalence.
\noproof
\end{corollary}

\begin{definition}
We define the {\em classifying space for free $P$-actions}, $EP$, to be any free $P$-space that is weakly contractible.
\end{definition}

\begin{proposition}
The space $EP$ exists, and is unique up $P$-homotopy equivalence.
\end{proposition}

\begin{proof}
By proposition \ref{ij}, a free and weakly contractible $P$-space $EP$ exists.  Let $X$ be another free and weakly contractible $P$-space.  Then by proposition \ref{weco}, the spaces $\Map_P (EP,X)$ and $\Map_P (X,EP)$ are weakly contractible.  In particular, they are non-empty, so we have  continuous equivariant maps $f\colon EP\rightarrow X$ and $g\colon X\rightarrow EP$.

Similarly, the spaces $\Map_P (EP,EP)$ and $\Map_P (X,X)$ are weakly contractible, so the sets  $\pi_0 \Map_P (EP,EP)$ and $\pi_0 (X,X)$ are trivial.  Since $g\circ f, \id_{EP}\in \Map_P (EP,EP)$, they must be $P$-homotopic.  Similarly, $f\circ g, \id_P \in \Map_P (X,X)$ are $P$-homotopic.  In other words, the composites $g\circ f$ and $f\circ g$ are both $P$-homotopic to identity maps, and we are done.
\end{proof}

\begin{definition}
Let $X$ be a $P$-space.  We define the {\em homotopy fixed point set} of $X$ to be the space $X^{hP} = \Map_P (EP,X)$.
\end{definition}

Up to homotopy, the space $X^{hP}$ does not depend on which version of the space $EP$ we have chosen.  Further, by proposition \ref{weco} and corollary \ref{maco}, if $X$ is weakly contractible, then so is $X^{hP}$, and if $f\colon X\rightarrow Y$ is a weak equivalence, then so is $f_\ast \colon X^{hP}\rightarrow Y^{hP}$.

Now, observe that if $X$ is a $P$-space, then the semigroup $P$ acts on the $C^\ast$-algebras $C^\ast (X)$ and $D^\ast (X)$ on the left by $\ast$-homomorphisms.  We can therefore form homotopy fixed point sets $C^\ast (X)^{hP}$ and $D^\ast (X)^{hP}$.  These sets are $C^\ast$-algebras, with addition, multiplication and involution defined pointwise, and the norm defined by taking the supremum
\[ \| f\| = \sup \{ \| f(x) \| \ |\ x\in EP \} \]
for $f\in C^\ast (X)^{hP}$ or $f\in D^\ast (X)^{hP}$.  Further, $C^\ast (X)^{hP}$ is a $C^\ast$-ideal in $D^\ast (X)^{hP}$, so we can form the quotient $D^\ast (X)^{hP}/C^\ast (X)^{hP}$.

Let us write
\[  K_n^{hP} (X) = K_{n+1} \left(  \frac{D^\ast (X)^{hP}}{C^\ast _P(X)^{hP}} \right) . \]

\begin{proposition} \label{hpg}
The sequence of functors $K_\ast^{hP}$ is a locally finite $P$-homology theory.
\end{proposition}

\begin{proof}
Let $U(X)$ be the stable unitary group of the $C^\ast$-algebra $D^\ast (X)/C^\ast (X)$.  Then the groups $K_{n-1}(X)$ and $K_{n-1}^{hP}(X)$ are, respectively, the homotopy groups of the groups $U(X)$ and $U(X)^{hP}$ respectively.

By proper homotopy-invariance of $K$-homology, the inclusions $i_0,i_1\colon X\rightarrow X\times [0,1]$ defined by the formulae $i_0 (x) = (x,0)$ and $i_1 (x) = (x,1)$ respectively, induce weak equivalences $U(X)\rightarrow U(X\times [0,1])$.  By corollary \ref{maco}, these maps both induce weak equivalences $U(X)^{hG} \rightarrow U(X\times [0,1])^{hG}$, and so isomorphisms $K_n^{hP}(X)\rightarrow K_n^{hP} (X\times [0,1])$.  Proper $P$-homotopy-invariance of the functors $K_n^{hP}$ now follows.

Let $X=A\cup B$ be a $P$-space, where we can write $A,B\subseteq X$ are open, and $PA\subseteq A$, $PB\subseteq B$.  Then by looking at Mayer-Vietories sequences in $K$-homology, we have a weak fibration sequence
\[ U(A\cap B) \rightarrow U(A)\vee U(B) \rightarrow U(X) \]
and so, by corollary \ref{maco}, a weak fibration sequence
\[ U(A\cap B)^{hP} \rightarrow U(A)^{hP}\vee U(B)^{hP} \rightarrow U(X)^{hP} . \]

The existence of Mayer-Vietoris sequences for the sequence of functors $K_\ast^{hP}$ now also follows.
\end{proof}

\section{Semigroup $C^\ast$-algebras and assembly}

Let $P$ be a semigroup.  Let $l^2 (P)$ be the Hilbert space with an orthonormal basis indexed by $P$, that is to say we have an orthonormal basis $\{ e_p \ |\ p\in P \}$.

Given $p\in P$, we have an isometry $v_p \colon l^2 (P)\rightarrow l^2 (P)$ defined by the formula $v_p (e_q) = e_{pq}$.  Note that for this to be an isometry, we need the left-cancellation property.

The following definition comes from \cite{Li}.

\begin{definition}
The {\em reduced semigroup $C^\ast$-algebra}, $C^\ast_r (P)$, is the smallest $C^\ast$-subalgebra of the $C^\ast$-subalgebra of the bounded linear operators ${\mathcal L}(l^2 (P))$ that contains the set of isometries $\{ v_p \ |\ p\in P \}$.
\end{definition}

Note that reduced group $C^\ast$-algebras are an obvious special case.

Now, let $X$ be a coarse $P$-space.  Then the $C^\ast$-algebra $C_0 (X)$ is equipped with a right $P$-action defined by writing $(fp)(x) = f(px)$ for all $f\in C_0(X)$, $p\in P$ and $x\in X$.

Let $H$ be a Hilbert space equipped with a left $P$-action by bounded linear maps.  Let ${\mathcal L}_P (H)$ be the $C^\ast$-algebra of equivariant bounded linear operators on $H$.  Then we call $H$ an {\em equivariant $H$-module} if it comes equipped with a $\ast$-homomorphism $\rho \colon C_0 (X)\rightarrow {\mathcal L}_P(H)$, 

\begin{definition}
Let $H$ be an ample equivariant $X$-module.  Then we define $D^\ast_P (X)$ to be the smallest $C^\ast$-subalgebra of ${\mathcal L}_P (H)$ containing all pseudolocal and controlled operators.

We define $C^\ast_P (X)$ to be the smallest $C^\ast$-subalgebra of ${\mathcal L}_P (H)$ containing all locally compact and controlled operators.
\end{definition}

\begin{theorem}
Let $P$ be a semigroup.  Let $X$ be a cocompact coarse $P$-space.  Then the $C^\ast$-algebras $C^\ast_r (P)$ and $C^\ast_P (X)$ are Morita equivalent.
\end{theorem}

\begin{proof}
Choose a compact subset $K\subseteq X$ such that $X=KP$.  Equip $K$ with a Borel measure.  Define a Hilbert space
\[ H = l^2 (P) \otimes L^2 (K) . \]

We have a left action of $P$ on $H$ defined by the formula
\[ p (v\otimes F) = pv \otimes F \qquad p\in P,\ v\in l^2 (P),\ F\in L^2 (K) . \]

Given $f\in C_0 (X)$ and $F\in L^2 (K)$, we have a function $f\cdot F \in L^2 (K)$ defined by pointwise multiplication.  We have a $\ast$-homomorphism $\rho_0 \colon C_0 (X)\rightarrow {\mathcal L}(H)$ defined by the formula
\[ \rho_0 (f) (v\otimes F ) = f(v) \otimes f\cdot F . \]

Let $M = H\otimes l^2 (\N )$.  Extend the $\ast$-homomorphism $\rho_0$ to a $\ast$-homomorphism $\rho \colon C_0 (X) \rightarrow {\mathcal L}(H)$ by acting trivially on the second factor.  Then $M$ is an ample $X$-module.

Observe that $C^\ast_r (P) \otimes {\mathcal K}(L^2 (K) \otimes l^2 (\N ))$ is a $C^\ast$-subalgebra of ${\mathcal L}(M)$, and any operator of the form $v_p \otimes k$, where $p\in P$ and $k$ is compact, is both controlled and locally compact.

Further, any controlled and locally compact operator on $M$ can be expressed as a finite sum of operators of this form.  Taking the $C^\ast$-completions, we see that
\[ C^\ast_r (P) \otimes {\mathcal K}(L^2 (K) \otimes l^2 (\N )) = C^\ast_P (X) \]
and we are done.
\end{proof}

Now, let $X$ be any $P$-space.  Then we have a short exact sequence
\[ 0 \rightarrow C^\ast_P (X) \rightarrow D^\ast_P (X) \rightarrow \frac{D^\ast_P (X)}{C^\ast _P(X)} \rightarrow 0 . \]

By the above, when $X$ is cocompact, we can identify the $K$-theory groups $K_\ast (C^\ast_P (X))$ and $K_\ast C^\ast_r (P)$.  Thus, looking at the boundary maps in the long exact sequence of $K$-theory groups (see for example \cite{RLL,W-O}), we obtain a map
\[ \beta \colon K_{\ast+1}  \left(  \frac{D^\ast_P (X)}{C^\ast _P(X)} \right) \rightarrow K_\ast C^\ast_r (P) \]
called the {\em analytic assembly map}.

This assembly map is a generalisation of the corresponding map for groups; see for example \cite{Roe1}.

\begin{definition}
Let $X$ be a $P$-space.  Then we define the $P$-equivariant $K$-homology groups of $X$ by writing
\[ K_n^P (X) = K_{n+1}  \left(  \frac{D^\ast_P (X)}{C^\ast _P(X)} \right) . \]
\end{definition}

\begin{proposition}
The sequence of functors $K_\ast^P$ defines a locally finite $P$-homology theory.
\end{proposition}

\begin{proof}
In general, for a $P$-space $Y$, let $Y^P = \{ y \in Y \ |\ py=y \textrm{ for all } p\in P \}$.  

Let $U(X)$ be the stable unitary group of the $C^\ast$-algebra $D^\ast (X)/C^\ast (X)$.  Then, by definition of the $C^\ast$-algebras $C^\ast_P (X)$ and $D^\ast_P (X)$ as fixed point $C^\ast$-algebras under a $P$-action by $\ast$-homomorphisms, the groups $K_{n-1}(X)$ and $K_{n-1}^{P}(X)$ are, respectively, the homotopy groups of the stable unitary groups $U(X)$ and $U(X)^{P}$ respectively.

The proof is now essentially the same as that of proposition \ref{hpg}.
\end{proof}

\begin{definition}
We say a torsion-free semigroup $P$ satisfies the {\em analytic Novikov conjecture} if the classifying space $EP$ is cocompact, and the map
\[ \beta \colon K_n^P (EP)\rightarrow K_n C^\ast_r (P) \]
is injective.
\end{definition}

We restrict our attention to torsion-free semigroups, since, in the case of groups, the map $\beta$ is not in general injective for groups with torsion, though it {\em is} conjectured to be rationally injective.  However, all of our arguments here are for torsion-free semigroups.

\section{Descent}

The {\em descent argument}, outlined in this section, tells us that the coarse Baum-Connes conjecture, along with certain mild extra conditions, implies the analytic Novikov conjecture. 

\begin{lemma}
We have a natural transformation $\theta_\ast \colon K_n^P (X) \rightarrow K_n^{hP}(X)$ that is an isomorphism whenever $X$ is a finite free $P$-$CW$-complex.
\end{lemma}

\begin{proof}
Let $U(X)$ and $U_P(X)$ be the stable unitary groups of the $C^\ast$-algebras $D^\ast (X)/C^\ast (X)$ and $D^ast_P (X)/C^\ast_P (X)$ respectively.  Since the space $EP$ is weakly contractible, we have a natural weak equivalence
\[ \begin{array}{rcl} 
j\colon U_P (X) & \simeq & \Map (EP,U_P(X)) \\
& = & \Map (EP,\Map_P(P,U_P(X))) \\
& = & \Map_P (EP,\Map (P,U_P(X))) \\
& = & \Map (P,E_P(X))^{hG} 
\end{array}. \]

Let $i\colon D^\ast_P (X)/C^\ast_P (X)\rightarrow D^\ast (X)/C^\ast (X)$ be defined by the inclusions $C^\ast_P (X)\hookrightarrow C^\ast (X)$, $D^\ast_P (X)\hookrightarrow D^\ast (X)$.  Then we have a natural map $k\colon \Map (P, U_P(X)) \rightarrow U(X)$ defined by writing $k(f) = i_\ast f(e)$, where $e$ is the identity element of the semigroup $P$.

Taking homotopy fixed point sets, we obtain a natural map $k' \colon \Map (P,U_P (X))^{hP} \rightarrow U(X)^{hP}$.  Composing with the map $j$, we have a natural map
\[ \theta =k'\circ j \colon U_P (X) \rightarrow U_P (X)^{hP} \]
and so a natural induced map  $\theta_\ast \colon K_n^P (X) \rightarrow K_n^{hP}(X)$.

Let $c\colon P\rightarrow +$ be the constant map onto the one point space.  Then the composition $c_\ast \circ i_\ast = i_\ast \circ c_\ast \colon U_P (P)\rightarrow U(+)$ is certainly a homotopy-equivalence, and the map $k \colon \Map (P,U_P(+))\rightarrow U(P)$ is a homeomorphism, and so a weak equivalence.  By corollary \ref{maco}, the map $k'$ is also a weak equivalence.

Thus the map $\alpha$ is a weak equivalence in this case, making the induced map $\alpha_\ast \colon K_n^P (P)\rightarrow K_n^{hP}(P)$ an isomorphism.

By theorem \ref{freeiso}, the map $\theta_\ast \colon K_n^P (X) \rightarrow K_n^{hP}(X)$ is therefore an isomorphism whenever $X$ is a finite free $P$-$CW$-complex.
\end{proof}

Now, let $X$ be a cocompact coarse $P$-space.  

Now, we can define a map $\eta_\ast \colon K_n (D^\ast_P(X))\rightarrow K_n (D^\ast (X)^{hP})$ in much the same way as the map $\theta_\ast$ in the above lemma, and so, whenever $X$ is a cocompact coarse $P$-space, we have a commutative diagram
\[ \begin{array}{ccccc}
K_n (D^\ast_P (X)) & \stackrel{v_G}{\rightarrow} & K_n^P (X) & \stackrel{\beta}{\rightarrow} & K_n C^\ast_r (P) \\
\downarrow & & \downarrow & & \\
K_n (D^\ast (X)^{hP}) & \stackrel{v_{hG}}{\rightarrow} & K_n^{hP} (X)  \\
\end{array} \]
where $\beta$ is the analytic assembly map.

\begin{theorem}
Let $X$ be a free coarse $P$-space that is a free finite $P$-$CW$-complex as a topological space.  Suppose the coarse Baum-Connes conjecture holds for $X$.  Then the analytic assembly map $\beta$ is injective for $X$.
\end{theorem}

\begin{proof}
The coarse Baum-Connes conjecture for $X$ implies that $K_n D^\ast (X) =0$ for all $n$.  Hence, by proposition \ref{weco}, $K_n (D^\ast (X)^{hP})=0$ for all $n$.

Now,  by the previous lemma, the map $\theta_\ast \colon K_n^P(X)\rightarrow K_n^{hp}(X)$ in the above commutative diagram is an isomorphism.

It follows that the map $v_p$ is zero, so the map $\beta$ is injective as required.
\end{proof}

\begin{corollary} \label{ANCD}
Let $P$ be a semigroup with a classifying space $EP$ that is a a coarse $P$-space and a finite $P$-$CW$-complex.  Suppose the coarse Baum-Connes conjecture holds for the space $EP$.

Then the analytic Novikov conjecture holds for the semigroup $P$.
\noproof
\end{corollary}

Note that if $G$ is a group with torsion, then the classifying space $BG$ is never a finite $CW$-complex, and so the universal cover $EG$ is never a finite $G$-$CW$-complex.   Thus, at least in the group case, the above result is not relevant when torsion is present.

\section{Examples}

We conclude this article by looking at some simple examples where the main result of the previous section applies.  As well as corollary \ref{ANCD}, we use the result from \cite{Yu2} that the coarse Baum-Connes conjecture holds for any bounded geometry coarse space which can be uniformly embedded in Hilbert space.

\subsection*{The group $\N$}

The group $\N$ acts freely on $\R^+$ by the writing $(n,x)\mapsto n+x$, where $n\in \N$, and $x\in \R^+$.  With the coarse structure defined by the metric, the space $\R^+$ is certainly uniformly embeddable in Hilbert space, so the coarse Baum-Connes conjecture holds.

Now, $\R^+$ is a finite free $\N$-$CW$-complex, with a single $0$-cell, $\N$, and $1$-cell $\N \times [0,1]$, with attaching map $f\colon \N \times \{ 0,1 \} \rightarrow \N$ defined by the formula $f(n,k)=n+k$.

Now $\R_+$ is weakly contractible, so we can take $E\N = \R_+$, and, by corollary \ref{ANCD}, the analytic Novikov conjecture holds for $\N$.

Similarly, let $\N^\times$ be the group of non-zero natural numbers with group operation defined by multiplication.  Then $E\N = [1,\infty )$, with free $\N$-action defined by writing $(n,x)\mapsto nx$.  As aboce, the analytic Novikov conjecture holds for $\N^\times$.

\subsection*{Free semigroups}

The {\em free semigroup on $n$ generators}, $V_n$, is the set of words in an alphebet with $n$ letters, say $e_1,\ldots ,e_n$.  Let $S= \{ e_1, \ldots ,e_n \}$.  Then $V_n$ certainly acts freely on the Cayley graph $\Cay (V_n; S)$, which is weakly contractible.

So we can take $EV_n = \Cay (V_n;S)$.  The space $EV_n$ is a finite $V_n$-$CW$-complex, with a single $0$-cell, the set of vertices, and $1$-cells $P\times \times \{ e_1 \} \times [0,1] , \ldots , P\times \{e_n \} \times [0,1]$.  The attaching map $f\colon P\times \{ e_i \} \times \{ 0,1\} \rightarrow P$ is defined by writing $f(p,e_i,0) =p$ and $f(p,e_i,1) = pe_i$.

The space $EV_n$ is certainly uniformly embeddable in an infinite-dimensional Hilbert space, so by corollary \ref{ANCD}, the analytic Novikov conjecture holds for the free semigroup $V_n$.

\subsection*{Products}

Let $P$ and $Q$ be semigroups such that $EP$ and $EQ$ are finite free $P$- and $Q$-$CW$-complexes respectively, and $EP$ and $EQ$ have compatible coarse structures where $P$ and $Q$ act respectively by coarse continuous maps.  Suppose $EP$ and $EQ$ are uniformly embeddable in Hilbert spaces $H_P$ and $H_Q$ respectively.

Then we can take $E(P\times Q)= EP\times EQ$.  The space $EP\times EQ$ is a free finite $P\times Q$-$CW$-complex, which is a coarse $P\times Q$-space uniformly embeddable in $H_P\times H_Q$.  Thus the analytic Novikov conjecture holds for $P\times Q$.

In particular, by the above, the analytic Novikov conjecture holds for the semigroup $\N^m \times (\N^\times )^n$ for all $m$ and $n$.

\subsection*{Subgroups}

Let $P$ be a semigroup with a classifying space $EP$ that is a a coarse $P$-space uniformly embeddable in Hilbert space, and a finite free $P$-$CW$-complex.  Let $P'$ be a subgroup of $P$.

Consider a $CW$-decomposition of $EP$.  Then by freeness, each cell takes the form $P\times D^n$.  We can therefore form a classifying space $EP'$ by replacing $P\times D^n$ by $P'\times D^n$.  Then $EP'$ is a finite free-$P'$-$CW$-complex, with a coarse structure such that it is a subspace of $EP$, which uniformly embeds in Hilbert space.

Thus the analytic Novikov conjecture holds for the subgroup $P'$.

\subsection*{The $ax+b$ semigroup over $\N$}

The {\em $ax+b$ semigroup over $\N$} is defined in \cite{Cu7} as the set
\[ \{ P_{\N } = \left\{ \left( \begin{array}{cc}
1 & k \\
0 & n \\ 
\end{array} \right) \ |\ n\in \N^\times ,\ k\in \N \right\} \]
with group operation defined by matrix multiplication.

The group $P_\N$ acts freely and cocompactly on the space $[0,\infty ) \times [1,\infty )$ by the formula
\[ \left( \begin{array}{cc}
1 & k \\
0 & n\\
\end{array} \right) \left( \begin{array}{c}
x \\
y \\
\end{array} \right) = \left( \begin{array}{c}
x+ky \\
ny \\
\end{array} \right) . \]

As above, the space $[0,\infty ) \times [1,\infty )$ has the structure of a finite $P_\N$-$CW$-complex.  As a coarse space, it is uniformly embeddable in Hilbert space.  So the analytic Novikov conjecture holds for $P_\N$.

\subsection*{Linear semigroups}

It is shown in \cite{Ji} that the group $GL (n, \Z )$ has finite asymptotic dimension with the word length metric, and $BGL (n,\Z )$ is a finite $CW$-complex.

Hence the universal cover, $EGL (n,\Z )$ is a finite $GL(n,\Z )$-$CW$-complex, and is coarsely equiivalent to $GL (n,\Z )$, so has finite asymptotic diminsion, and therefore uniformly embeds in Hilbert space.

Thus the analytic Novikov conjecture holds for $GL (n,\Z )$, and for any subsemigroup by the above.  In particular, it holds for $GL (n, \N )$.


\begin{thebibliography}{10}

\bibitem{BCH}
P.~Baum, A.~Connes, and N.~Higson.
\newblock {Classifying spaces for proper actions and $K$-theory of group
  $C^\ast$-algebras}.
\newblock In S.~Doran, editor, {\em $C^\ast$-algebras: 1943--1993}, volume 167
  of {\em Contemporary Mathematics}, pages 241--291. American Mathematical
  Society, 1994.

\bibitem{Cu7}
J.~Cuntz.
\newblock {$C^\ast$ algebras associated with the $ax+b$ semigroup over $\mathbb
  N$}.
\newblock In {\em $K$-theory and Noncommutative geometry}, EMS series of
  congress reports, pages 201--216. European Mathematical Society, 2008.

\bibitem{HLS}
N.~Higson, V.~Lafforgue, and G.~Skandalis.
\newblock {Counterexamples to the Baum-Connes conjecture}.
\newblock {\em Geometric and Functional Analysis}, 12:330--354, 2002.

\bibitem{HR1}
N.~Higson and J.~Roe.
\newblock {\em {Analytic $K$-homology}}.
\newblock Oxford Mathematical Monographs. Oxford University Press, 2000.

\bibitem{Ji}
L.~Ji.
\newblock {Asymptotic dimension and the integral $K$-theoretic Novikov
  conjecture for arithmetic groups}.
\newblock {\em Journal of Differential Geometry}, 68:535--544, 2004.

\bibitem{Li}
X.~Li.
\newblock {Semigroup $C^\ast$-algebras and amenability of semigroups}.
\newblock ArXiV: 1055.5539v1, 2011.

\bibitem{Mil6}
J.~Milnor.
\newblock {Construction of universal bundles, II}.
\newblock {\em Annals of Mathematics}, 63:430--436, 1956.

\bibitem{Mitch13}
P.D. Mitchener.
\newblock The general notion of descent in coarse geometry.
\newblock {\em Algebraic and Geometric topology}, 10:2149--2450, 2010.

\bibitem{Roe1}
J.~Roe.
\newblock {\em Index Theory, Coarse Geometry, and Topology of Manifolds},
  volume~90 of {\em CBMS Regional Conference Series in Mathematics}.
\newblock American Mathemtaical Society, 1996.

\bibitem{Roe6}
J.~Roe.
\newblock {\em Lectures on Coarse Geometry}, volume~31 of {\em University
  Lecture Series}.
\newblock American Mathematical Society, 2003.

\bibitem{RLL}
M.~R{\o}rdam, F.~Larsen, and N.~Laustsen.
\newblock {\em An introduction to $K$-theory for $C^\ast$-algebras}.
\newblock Number~49 in London Mathematical Society Student Texts. Cambridge
  University Press, 2000.

\bibitem{W-O}
N.E. Wegge-Olsen.
\newblock {\em {$K$-theory and $C^\ast$-algebras}}.
\newblock Oxford Science Publications. Oxford University Press, 1994.

\bibitem{Yu2}
G.L. Yu.
\newblock {The coarse Baum-Connes conjecture for spaces which admit a uniform
  embedding into Hilbert space}.
\newblock {\em Inventiones Mathematicae}, 139:201--240, 2000.

\end{thebibliography}
\end{document}